\documentclass{amsart}
\usepackage{amsmath}
\allowdisplaybreaks

\usepackage{amssymb}
\theoremstyle{proclaim}
\newtheorem{theorem}{Theorem}[section]
\newtheorem{lemma}[theorem]{Lemma}
\newtheorem{corollary}[theorem]{Corollary}
\newtheorem{proposition}[theorem]{Proposition}
\theoremstyle{statement}
\newtheorem{remark}[theorem]{Remark}
\newtheorem{definition}[theorem]{Definition}

\numberwithin{equation}{section}

\newcommand{\Z}{\mathbb{Z}}

\newcommand{\OO}{\mathcal{O}}
\newcommand{\iO}{\mathcal{JO}}
\newcommand{\F}{\mathfrak{F}}
\newcommand{\PP}{\mathbb{P}}
\newcommand{\C}{\mathbb{C}}
\newcommand{\B}{\mathbb{B}}

\newcommand{\R}{\mathbb{R}}
\newcommand{\T}{\mathbb{T}}
\newcommand{\bS}{\mathbb{S}}
\newcommand{\mA}{\mathcal{A}}
\newcommand{\mR}{\mathcal{R}}
\newcommand{\mH}{\mathcal{H}}
\newcommand{\mHR}{\mathcal{HR}}
\newcommand{\mT}{\mathcal{T}}
\newcommand{\PC}{\PP^n(\C)}
\newcommand{\SU}{\mathrm{SU}}
\newcommand{\PSU}{\mathrm{PSU}}
\newcommand{\SL}{\mathrm{SL}}
\newcommand{\PSL}{\mathrm{PSL}}
\newcommand{\Lie}{\mathrm{Lie}}
\newcommand{\sli}{\mathfrak{sl}}

\begin{document}

\title[Quasi-radial quasi-homogeneous symbols and Lagrangian
frames]{Toeplitz operators with quasi-radial quasi-homogeneous symbols
  and bundles of Lagrangian frames}

\author[Raul Quiroga-Barranco {\protect \and} Armando
Sanchez-Nungaray]{Raul Quiroga-Barranco {\protect \and} Armando Sanchez-Nungaray}
\address{Raul Quiroga-Barranco, Centro de Investigaci\'on en
  Matem\'aticas, Guanajuato, M\'exico}
\email{quiroga@cimat.mx}
\address{Armando Sanchez-Nungaray, Centro de Investigaci\'on en
  Mate\-m\'aticas, Guanajuato, M\'exico}
\email{armandos@cimat.mx}

\begin{abstract}
  We prove that the quasi-homogenous symbols on the projective space
  $\PC$ yield commutative algebras of Toeplitz operators on all
  weighted Bergman spaces, thus extending to this compact case known
  results for the unit ball $\B^n$. These algebras are Banach but not
  $C^*$. We prove the existence of a strong link between such symbols
  and algebras with the geometry of $\PC$.
\end{abstract}



\thanks{The first named author was partially supported by SNI-Mexico
  and by a Conacyt grant. The second named author was partially
  supported by a Conacyt postdoctoral fellowship.}

\maketitle

\section{INTRODUCTION}
The study of commutative algebras of Toeplitz operators has shown to
be a very interesting subject. Some previous results in this topic
serve as background to this work. First, it was shown the existence of
symbols defining interesting commutative $C^*$-algebras of Toeplitz
operators on bounded symmetric domains (see \cite{GQV-disk},
\cite{QV-Reinhardt}, \cite{QV-Ball1} and \cite{QV-Ball2}).  Also, it
was exhibited in \cite{NikolaiQuasi} the existence of Banach algebras,
which are not $C^*$, of commutative Toeplitz operators on the unit
ball $\B^n$. And, in \cite{QS-Proj} we constructed commutative
$C^*$-algebras of Toeplitz operators on complex projective spaces.

A remarkable fact is that the currently known commutative
$C^*$-algebras of Toeplitz operators on $\B^n$ are naturally
associated to Abelian subgroups of the group of biholomorphisms of
$\B^n$. In fact, their systematic description is best understood with
the use of such groups of biholomorphisms (see \cite{QV-Ball1} and
\cite{QV-Ball2}). Furthermore, this provided the guiding light to
construct commutative $C^*$-algebras of Toeplitz operators in the
complex projective space $\PC$: the currently known $C^*$-algebras for
$\PC$ are naturally associated and described from the maximal tori of
the group of isometric biholomorphisms of $\PC$ (see \cite{QS-Proj}).

The known Banach (not $C^*$) algebras of commutative Toeplitz
operators first introduced in \cite{NikolaiQuasi} for $\B^n$ are given
by the so called quasi-homogeneous symbols. Such symbols are defined
in terms of radial and spherical coordinates of components in $\B^n$
(see Section~\ref{sec:quasi} below). However, their introduction
lacked the stronger connection with the geometry of the domain
observed for the commutative $C^*$-algebras of Toeplitz operators on
$\B^n$.

Given these lines of research, there are two natural problems to
consider. First, to determine whether or not there are any interesting
Banach algebras, that are not $C^*$, of commutative Toeplitz operators
on $\PC$. Second, assuming that such Banach algebras exist, to find
any possible special links between the geometry of $\PC$ and those
Banach algebras of commutative Toeplitz operators. The goal of this
work is to solve these problems.

On one hand, we define quasi-homogeneous symbols in the complex
projective space $\PC$, and show that these provide Banach algebras of
commutative Toeplitz operators on every weighted Bergman space of
$\PC$. The results that exhibit the commuting Toeplitz operators are
obtained in Section~\ref{sec:commutativity}, where
Theorem~\ref{thm:Banachalgebra} is the main result. On the other hand,
we also prove that the Banach algebras defined by quasi-homogeneous
symbols turn out to have a strong connection with the geometry of the
supporting space $\PC$; the main results in this case are presented in
Section~\ref{sec:Lagrangian}. In particular, we prove that the
quasi-homogeneous symbols on $\PC$ can be associated to an Abelian
group of holomorphic isometries of the corresponding space (see
Theorem~\ref{thm:torus-mA}). Such group is a subgroup of a maximal
torus in the corresponding isometry group.

We further prove that the groups associated to quasi-homogeneous
symbols afford pairs of foliations with distinguished Lagrangian and
Riemannian geometry known as Lagrangian frames (see
Section~\ref{sec:Lagrangian} below and \cite{QV-Reinhardt},
\cite{QV-Ball1} and \cite{QV-Ball2}). This recovers the behavior
observed for the $C^*$-algebras of commutative Toeplitz operators
constructed in \cite{QV-Ball1}, \cite{QV-Ball2} and \cite{QS-Proj},
for which such Lagrangian frames appear as well. Nevertheless, it is
important to note a key difference between the $C^*$ case and the
Banach case. For the $C^*$-algebras of Toeplitz operators on $\B^n$
and $\PC$, as constructed in \cite{QV-Ball1}, \cite{QV-Ball2} and
\cite{QS-Proj}, the Lagrangian frames are obtained for the whole
space, i.e.~they come from Lagrangian submanifolds of the whole space,
either $\B^n$ or $\PC$. But for the Banach algebras given by the
quasi-homogeneous symbols considered here the Lagrangian frames are
obtained on submanifolds of $\PC$ that provide both an stratification
and a partition into principal fiber bundles. The existence of the
principal bundles and the Lagrangian frames on suitable submanifolds
is obtained in Theorems~\ref{thm:partitions-bundles} and
\ref{thm:Lframes}, respectively. Nevertheless, as in the case of
$C^*$-algebras, the Riemannian leaves of the Lagrangian frames that we
exhibit are precisely the orbits of the Abelian group associated to
the symbols. Also, it is proved in Theorem~\ref{thm:bundleaut} that a
full maximal torus of isometries continues to play an importante role,
since the complement to the group that defines the fiberwise
Lagrangian frames acts by automorphisms of such frames.

We want to acknowledge the input of Pedro Luis del Angel, from Cimat,
provided through several conversations with him. We also want to thank
an anonymous referee whose remarks allowed us to greatly improve a
previous version of this work.

\section{PRELIMINARIES ON THE GEOMETRY AND ANALYSIS OF
  $\PC$}\label{sec:prel-PC}

In this section we establish our notation concerning the
$n$-dimensional complex projective space $\PC$. We will freely use the
well known properties of the projective space and refer to the
bibliography for further details. In particular, we recall that for $w
\in \C^{n+1}\setminus \{0\}$ the element $[w] \in \PC$ is said to have
homogeneous coordinates given by $w$.

There is a natural realization of $\C^n$ as an open conull dense
subset given by
\[
  \C^n \rightarrow \PC, \quad  z \mapsto [1,z],
\]
which defines a biholomorphism onto its image. Note that the points of
$\PC$ are denoted by $[w]$, the complex line through $w \in
\C^{n+1}\setminus\{0\}$. We will refer to this embedding as the
canonical embedding of $\C^n$ into $\PC$.

Let us denote by $\omega$ the canonical K\"ahler structure on $\PC$
that defines the Fubini-Study metric, whose volume is then given by
$\Omega = (\omega/2\pi)^n$. These induce on $\C^n$ the following
K\"ahler form and volume element, respectively
\begin{align*}
  \omega_0 &= i\frac{(1+|z|^2) \sum_{k=1}^n dz_k\wedge d\overline{z}_k
    - \sum_{k,l=1}^n \overline{z}_k z_l
    dz_k\wedge d\overline{z}_l}{(1+|z|^2)^2}, \\
  \Omega_0 &= \frac{1}{\pi^n} \frac{dV(z)}{(1 + |z_1|^2 + \dots +
    |z_n|^2)^{n+1}},
\end{align*}
where $dV(z)$ denotes the Lebesgue measure on $\C^n$.

Let $H$ denote the dual bundle of the tautological line bundle of
$\PC$. We recall that $H$ carries a canonical Hermitian metric $h$
obtained from the (flat) Hermitian metric of $\C^{n+1}$. Then, it is
also well known that the curvature $\Theta$ of $(H,h)$ satisfies the
identity
\[
\Theta = -i \omega,
\]
which amounts to say that $(H,h)$ is a quantum line bundle over $\PC$.

We will denote by $\Gamma(\PC, H^m)$ and $\Gamma_{hol}(\PC, H^m)$ the
smooth and holomorphic sections of $H^m$, respectively. Note that
$H^m$ denotes the $m$-th tensorial power of $H$. Clearly, both of
these spaces lie inside $L_2(\PC,H^m)$.

For every $m\in \Z_+$ and with respect to the canonical embedding of
$\C^n$ into $\PC$, we define the weigthed measure on $\PC$ with weight
$m$ by
\begin{align*}
d\nu_m(z) &= \frac{(n+m)!}{m!}\frac{\Omega(z)}{(1 + |z_1|^2 + \dots +
  |z_n|^2)^m} \\
&= \frac{(n+m)!}{\pi^n m!} \frac{dV(z)}{(1 + |z_1|^2 +
  \dots + |z_n|^2)^{n+m+1}}.
\end{align*}
A simple computation shows that $d\nu_m$ is a probability measure for
all $m \in \Z_+$.  For simplicity, we will use the same symbol
$d\nu_m$ to denote the weighted measures for both $\PC$ and $\C^n$.
It is also straightforward to show that the canonical embedding of
$\C^n$ into $\PC$ induces a canonical isometry
\[
\Phi : L_2(\PC, H^m) \rightarrow L_2(\C^n, \nu_m)
\]
with respect to which we will identify these spaces in the rest of
this work. Also, we will denote by $\left<\cdot,\cdot\right>_m$ the
inner product of this Hilbert spaces.

The weighted Bergman space on $\PC$ with weight $m\in \Z_+$ is defined
by:
\begin{align*}
  \mA^2_m(\PC) &= \{ \zeta \in L_2(\PC,H^m) : \zeta \text{ is
    holomorphic}
  \}\\
  &= \Gamma_{hol}(\PC,H^m).
\end{align*}

These Bergman spaces are finite-dimensional and are described by the
following well known result.

\begin{proposition}\label{prop:sections_polys}
  For every $m \in \Z_+$, the Bergman space $\mA^2_m(\PC)$ satisfies
  the following properties.
  \begin{enumerate}
  \item $\mA^2_m(\PC)$ can be identified with the space
    $P^{(m)}(\C^{n+1})$ of homogeneous polynomials of degree $m$ over
    $\C^{n+1}$.
  \item For $\Phi : L_2(\PC, H^m) \rightarrow L_2(\C^n, \nu_m)$ the
    canonical isometry described above, we have $\Phi(\mA^2_m(\PC)) =
    P_m(\C^n)$, the space of polynomials on $\C^n$ of degree at most
    $m$.
  \end{enumerate}
\end{proposition}

In what follows, we will use this realization of the Bergman spaces
without further notice.

Recall the following notation for multi-indices $\alpha, \beta \in
\Z_+^n$ and $z \in \C^n$
\begin{align*}
  |\alpha| &= \alpha_1 + \dots + \alpha_n, \\
  \alpha! &= \alpha_1! \dots \alpha_n!,\quad \alpha \in \Z_+^n, \\
  z^\alpha &= z^{\alpha_1}\dots z^{\alpha_n}, \\
  \delta_{\alpha,\beta} &= \delta_{\alpha_1,\beta_1} \dots
  \delta_{\alpha_n,\beta_n}.
\end{align*}

The Bergman space $\mA^2_m(\PC)$ has a basis consisting of the
polynomials $z^{\alpha}=z_1^{\alpha_1}\dots z_n^{\alpha_n}$ where
$\alpha \in \Z_+^n$ and $|\alpha| \leq m$. Hence we will consider the
set
\[
J_n(m) =\{ \alpha\in\Z_+^n : |\alpha| \leq m \}.
\]
More precisely, an easy computation shows that the set
\begin{equation}\label{eq:orth-basis}
\left\{\left(\frac{m!}{\alpha!
      (m-|\alpha|)!}\right)^{\frac{1}{2}}z^{\alpha}: \alpha\in J_n(m)
\right\}
\end{equation}
is an orthonormal basis of $\mA^2_m(\PC)$.

For $\psi\in L_2(\PC,H^m)$, and considering the identification $\Phi$,
we define the Bergman projection by
\[
B_m(\psi)(z)=\frac{(n+m)!}{\pi^n m!} \int_{\C^n} \frac{\psi(w)K(z,w)
  dV(w)}{(1+w_1\overline{w}_1+\cdots +w_n\overline{w}_n)^{n+m+1}}
\]
where
\[
K(z,w)=(1+z_1\overline{w}_1\cdots +z_n\overline{w}_n)^m.
\]
Note that the above integral is a polynomial of degree at most $m$ in
the variable $z$, and so it defines an element of $\mA^2_m(\PC)$ with
respect to the identification $\Phi$.  Also, the operator $B_m$
satisfies the well known reproducing property.

\begin{proposition}
  If $\psi \in L_2(\PC,H^m)$, then $B_m(\psi)$ belongs to the weighted
  Bergman space $\mA^2_m(\PC)$. Also, $B_m(\psi)=\psi$ if $\psi\in
  \mA^2_m(\PC)$.
\end{proposition}

Using this, we define the Toeplitz operator $T_a$ on $\mA^2_m(\PC)$
with bounded symbol $a \in L_\infty(\PC)$ by
\[
T_a(\varphi) = B_m(a \varphi),
\]
for every $\varphi \in \mA^2_m(\PC)$.

Let us denote by $\bS^n$ the unit sphere in $\C^n$.  In this work, we
will use the following identity on the sphere $\bS^n$
\begin{equation}\label{int-sphere}
  \int_{\bS^n} \xi^{\alpha} \overline{\xi}^{\beta} dS(\xi)=
  \delta_{\alpha,\beta}\frac{2\pi^n \alpha!}{(n-1+|\alpha|)!}
\end{equation}
where $dS$ be the corresponding surface measure on $\bS^n$ (see
\cite{Zhu}).

\section{TOEPLITZ OPERATORS WITH QUASI-HOMOGENEOUS SYMBOLS}
\label{sec:quasi}
Quasi-homogeneous symbols were introduced in \cite{NikolaiQuasi} on
the unit ball $\B^n$ in $\C^n$. The same sort of symbols can be
defined on the complex projective space $\PC$ using the homogeneous
coordinates of its elements.

Let $k=(k_0,\dots,k_l) \in \Z_+^{l+1}$ be a multi-index so that $|k| =
n+1$. We will call such multi-index $k$ a partition of $n+1$. For the
sake of definiteness, we will always assume that $k_0 \leq \dots \leq
k_l$. This partition provides a decomposition of the coordinates $w
\in \C^{n+1}$ as $w = (w_{(0)}, \dots, w_{(l)}$) where
\[
w_{(j)} = (w_{k_0+\dots+k_{j-1}+1}, \dots, w_{k_0+\dots+k_{j}}),
\]
for every $j = 0, \dots, l$, and the empty sum is $0$ by
convention. For $w \in \C^{n+1}$, we define $r_{j} = |w_{(j)}|$ and
\[
\xi_{(j)} = \frac{w_{(j)}}{r_{j}}
\]
if $w_{(j)} \not= 0$. Besides the quasi-radii $(r_{0}, \dots, r_{l})$,
this provides a set of coordinates $(\xi_{(0)}, \dots, \xi_{(l)}) \in
\bS^{k_{0}} \times \dots \times \bS^{k_{l}}$.

\begin{definition}
  \label{def:quasi-homogeneous}
  Let $k = (k_0, \dots, k_l) \in \Z_+^{l+1}$ be a partition of $n+1$
  and let $p,q \in \Z_+^{n+1}$ be such that
  \[
  p \cdot q = p_0q_0 + \dots + p_nq_n = 0, \quad |p| = |q|.
  \]
  With the above notation, the $k$-quasi-homogeneous symbol associated
  to $p,q$ is the function $\varphi : \PC \rightarrow \C$ given by
  \[
  \varphi([w]) = \xi^p\overline{\xi}^q = \prod_{j=0}^l
  \left(\frac{w_{(j)}}{|w_{(j)}|}\right)^{p_{(j)}}
  \left(\frac{\overline{w}_{(j)}}{|w_{(j)}|}\right)^{q_{(j)}} =
  \prod_{j=0}^l \left(\frac{w_{(j)}}{r_{j}}\right)^{p_{(j)}}
  \left(\frac{\overline{w}_{(j)}}{r_{j}}\right)^{q_{(j)}}.
  \]
  We will denote by $\mH_k$ the set of $k$-quasi-homogeneous symbols
  on $\PC$.
\end{definition}

It is a simple exercise to prove that the condition $|p| = |q|$
implies that the function $\varphi$ from
Definition~\ref{def:quasi-homogeneous} is well defined, i.e.~its
expression is independent of the choice of homogeneous
coordinates. Furthermore, for the canonical embedding $\C^n
\hookrightarrow \PC$ the symbols from
Definition~\ref{def:quasi-homogeneous} have as a particular case the
symbols considered in \cite{NikolaiQuasi}. The latter is the content
of the following easy to prove result.

\begin{lemma}
  \label{lem:quasi-homog-PC-Unit}
  Let $k \in \Z_+^{l+1}$ be a partition of $n+1$ of the form $k = (k_0
  = 1, k_1, \dots, k_l)$, so that $k' = (k_1, \dots, k_l)$ is a
  partition of $n$. If $p,q \in \Z_+^{n+1}$ satisfy $p\cdot q = 0$ and
  $|p| = |q|$, then with respect to the canonical embedding $\C^n
  \hookrightarrow \PC$ the $k$-quasi-homogeneous symbol $\varphi \in
  \mH_k$ associated to $p,q$ restricted to $\C^n$ satisfies
  \[
  \varphi([1,z]) = \prod_{j=1}^l
  \left(\frac{z_{(j)}}{|z_{(j)}|}\right)^{p_{(j)}}
  \left(\frac{\overline{z}_{(j)}}{|z_{(j)}|}\right)^{q_{(j)}}.
  \]
  In particular, $\varphi$ restricted to $\B^n \subset \C^n \subset
  \PC$ is a $k'$-quasi-homogeneous symbol in the sense of
  \cite{NikolaiQuasi}.
\end{lemma}

As a consequence, the quasi-homogeneous symbols on $\PC$ from
Definition~\ref{def:quasi-homogeneous} correspond to those considered
in \cite{NikolaiQuasi} for the unit ball.

\begin{definition}
  \label{def:quasi-radial}
  Let $k = (k_0, \dots, k_l) \in \Z_+^{l+1}$ be a partition of
  $n+1$. With the above notation, a $k$-quasi-radial symbol is a
  function function $a : \PC \rightarrow \C$ that can be written in
  the form $a([w]) = \widetilde{a}(|w_{(0)}|, \dots, |w_{(l)}|)$ for
  some function $\widetilde{a} : [0,+\infty)^{l+1} \rightarrow \C$
  which is homogeneous of degree $0$. We will denote by $\mR_k$ the
  set of $k$-quasi-radial symbols on $\PC$.
\end{definition}

Note that the degree $0$ homogeneity condition ensures that such
quasi-radial symbols are well defined. Also, the following obvious
result shows that suitable quasi-radial symbols restrict to those
defined in \cite{NikolaiQuasi}.

\begin{lemma}
  \label{lem:quasi-radial-PC-unit}
  Let $k \in \Z_+^{l+1}$ be a partition of $n+1$ of the form $k = (k_0
  = 1, k_1, \dots, k_l)$, so that $k' = (k_1, \dots, k_l)$ is a
  partition of $n$. Then, with respect to the canonical embedding
  $\C^n \hookrightarrow \PC$, every $k$-quasi-radial symbol $a \in
  \mR_k$ restricted to $\C^n$ defines a function $\C^n \rightarrow \C$
  that depends only $|z_{(1)}|, \dots, |z_{(l)}|$, where $z \in
  \C^n$. In particular, the symbol $a$ restricted to $\B^n \subset
  \C^n \subset \PC$ is a $k'$-quasi-radial symbol in the sense of
  \cite{NikolaiQuasi}.
\end{lemma}

By putting together both definitions above, we obtain the notion of
quasi-homogeneous quasi-radial symbol.

\begin{definition}
  \label{def:quasi-homog-radial}
  Let $k = (k_0, \dots, k_l) \in \Z_+^{l+1}$ be a partition of
  $n+1$. Then, a $k$-quasi-homogeneous quasi-radial symbol is a
  function $\PC \rightarrow \C$ of the form $a\varphi$, where $a \in
  \mR_k$ and $\varphi \in \mH_k$; in this case, we will refer to $a$
  and $\varphi$ as the quasi-radial and quasi-homogeneous parts of the
  symbol, respectively. We will denote by $\mHR_k$ the set of
  $k$-quasi-homogeneous quasi-radial symbols.
\end{definition}

As a consequence of Lemmas~\ref{lem:quasi-homog-PC-Unit} and
\ref{lem:quasi-radial-PC-unit}, for a partition of $n+1$ of the form
$k = (k_0 = 1, k_1, \dots, k_l)$ the symbols in $\mHR_k$ restrict to
functions on $\B^n \subset \C^n \subset \PC$ that are
quasi-homogeneous quasi-radial symbols in the sense of
\cite{NikolaiQuasi}, with the latter corresponding to the partition
$k' = (k_1, \dots, k_l)$ of $n$. On the other hand, it turns out that
the condition $k_0 = 1$ is not very restrictive. In fact, the
following result shows that assuming $k_0 = 1$ already provides the
more general notion of quasi-homogeneous quasi-radial symbol.

\begin{lemma}
  \label{lem:k0=1}
  Let $k = (k_0, \dots, k_l) \in \Z_+^{l+1}$ be a partition of $n+1$
  and assume that $k_0 > 1$. If we consider the partition
  $\widetilde{k} = (1, k_0-1, k_1, \dots, k_l)$, then we have $\mHR_k
  \subset \mHR_{\widetilde{k}}$. In particular, every element of
  $\mHR_k$ when restricted to $\B^n \subset \C^n \subset \PC$ is a
  quasi-homogeneous quasi-radial symbol in the sense of
  \cite{NikolaiQuasi}.
\end{lemma}
\begin{proof}
  For every element $w \in \C^{k_0}$ we will write $w = (w_1, w')$
  where $w_1 \in \C$ and $w' \in \C^{k_0-1}$. In particular, for every
  $w \in \C^{k_0}$ we have $|w| = \sqrt{|w_1|^2 + |w'|^2}$, which
  clearly implies the inclusion $\mR_k \subset \mR_{\widetilde{k}}$.

  On the other hand, for $w \in \C^{k_0}$ and $p \in \Z_+^{k_0}$ we
  have the identity
  \[
  \left(\frac{w}{|w|}\right)^p =
  \frac{|w_1|^{p_1}|w'|^{|p'|}}{(|w_1|^2 + |w'|^2)^{\frac{|p|}{2}}}
  \left(\frac{w_1}{|w_1|}\right)^{p_1}
  \left(\frac{w'}{|w'|}\right)^{p'}.
  \]
  We observe that the first fraction in the right-hand side of the
  identity is quasi-radial with respect to the partition
  $(1,k_0-1)$. Hence, it is easy to see that the last identity implies
  the inclusion $\mH_k \subset \mHR_{\widetilde{k}}$. Since
  $\mHR_{\widetilde{k}}$ is clearly closed under multiplication of
  functions, the result follows from these remarks.
\end{proof}

As a consequence of this result, without loss of generality in the
study of quasi-homogeneous quasi-radial symbols we will assume that
every partition of the homogeneous coordinates of $\PC$ is of the form
$(1, k_1, \dots, k_l)$, where $k = (k_1, \dots, k_l) \in \Z_+^l$ is
thus a partition of $n$. In other words, it is enough to study the set
of symbols $\mHR_{(1,k)}$ where $k \in \Z_+^l$ is a partition of $n$.
Furthermore, once we have such assumption,
Lemmas~\ref{lem:quasi-homog-PC-Unit} and
\ref{lem:quasi-radial-PC-unit} provide the expression for the symbols
on $\C^n \subset \PC$.

Also, since the local chart given by the canonical embedding $\C^n
\hookrightarrow \PC$ covers a conull subset of $\PC$, every
computation involving integrals can be performed on $\C^n$. We will
make use of such simplification in the rest of this work.

\begin{remark}
  \label{rmk:Cn}
  Let $k = (k_1, \dots, k_l) \in \Z_+^l$ be a partition of $n$. In the
  rest of this work, if a symbol $\xi^p \overline{\xi}^q$ in
  $\mH_{(1,k)}$ is considered as a function $\C^n \rightarrow \C$,
  then we will assume that it is an expression of the form
    \[
    \xi^p \overline{\xi}^q =
    \prod_{j=1}^l
    \left(\frac{z_{(j)}}{|z_{(j)}|}\right)^{p_{(j)}}
    \left(\frac{\overline{z}_{(j)}}{|z_{(j)}|}\right)^{q_{(j)}},
    \]
    where $\xi \in \bS^{k_1} \times \dots \bS^{k_l}$ is given by
    $\xi_{(j)} = \frac{z_{(j)}}{|z_{(j)}|}$, for $z = (z_{(1)}, \dots,
    z_{(l)}) \in \C^n$, and $p,q \in \Z_+^n$. In particular, such a
    symbol $\xi^p \overline{\xi}^q$ satisfies
    Definition~\ref{def:quasi-homogeneous} for the exponents $(0,p),
    (0,q) \in \Z_+^{n+1}$ and the homogeneous coordinates of $\PC$.
\end{remark}

Observe that for every partition $k \in \Z_+^l$ of $n$ the family
$\mR_{(1,k)}$ is contained in $\mR_{(1,...,1)}$. Hence, the Toeplitz
operators $T_a$ for symbols $a\in \mR_{(1,k)}$ can be simultaneously
diagonalized with respect to the monomial basis in the corresponding
Bergman space (see \cite{QV-Ball1} and \cite{QS-Proj}). Furthermore,
the following result provides the multiplication operator so obtained.

\begin{lemma}\label{Lemma-kradial}
  Let $k = (k_1, \dots, k_l) \in \Z_+^l$ be a partition of $n$, and
  let $a \in \mR_{(1,k)}$ be a $(1,k)$-quasi-radial bounded measurable
  symbol considered as a function $\C^n \rightarrow \C$. Let $T_a$ be
  the Toeplitz operator defined by $a$ on $\mA^2_m(\PC) \simeq
  P_m(\C^n)$ (identified by
  Proposition~\ref{prop:sections_polys}). Then $T_a
  z^{\alpha}=\gamma_{a,k,m}(\alpha)z^{\alpha}$, for every $\alpha\in
  J_n(m)$, where
  \begin{align}\label{spectrum-radial}
    \nonumber
    \gamma_{a,k,m}(\alpha)&=\gamma_{a,k,m}(|\alpha_{(1)}|,\dots,
    |\alpha_{(l)}|) \\
    &=\frac{2^l  (n+m)!}{(m-|\alpha|)!
      \prod_{j=1}^{l}(k_j-1+|\alpha_{(j)}|)} \\
    \nonumber
    &\times \int_{\R^n_+}  a(r_1,\dots,r_l)
    (1+r^2)^{-(n+m+1)}\prod_{j=1}^{l}r_j^{2|\alpha_{(j)}|+2k_j-1} dr_j.
  \end{align}
\end{lemma}
\begin{proof}
  Let $\alpha\in \Z^n_+$ with $|\alpha| \leq m$. Then, we have
  \begin{align*}
    \langle T_a z^{\alpha}, z^{\alpha}\rangle_m
    &=\langle a z^{\alpha}, z^{\alpha}\rangle_m \\
    &=\frac{(n+m)!}{\pi^n m!} \int_{\C^n}
    \frac{a(r_1,\dots,r_l) z^{\alpha}\overline{z}^\alpha
      dV(z)}{(1+|z_1|^2 + \dots + |z_n|^2)^{n+m+1}}.
  \end{align*}
  Making the substitution $z_{(j)} = r_j \xi_{(j)}$, where $r_j \in
  [0, \infty)$ and $\xi_{(j)} \in \bS^{k_j}$, for $j =
  1,\dots,l$, we obtain
  \begin{align*}
    \langle a z^\alpha, z^\alpha\rangle_{m}
    &=\frac{ (n+m)!}{\pi^n m!} \int_{\R^n_+} a(r_1,\dots,r_l)
    (1+r^2)^{-(n+m+1)} \prod_{j=1}^{l}r_j^{2|\alpha_{(j)}|+2k_j-1} dr_j \\
    &\times \prod_{j=1}^{l} \int_{\bS^{k_j}}
    \xi^{\alpha_{(j)}}_{(j)}\overline{\xi}^{\alpha_{(j)}}_{(j)}
    dS(\xi_{(j)})  \\
    &=\frac{2^l \alpha! (n+m)!}{ m! \prod_{j=1}^{l}(k_j-1+|\alpha_{(j)}|)!}\\
    &\times \int_{\R^n_+}  a(r_1,\dots,r_l) (1+r^2)^{-(n+m+1)}
    \prod_{j=1}^{l}r_j^{2|\alpha_{(j)}|+2k_j-1} dr_j
  \end{align*}
  and the result follows from (\ref{int-sphere})
\end{proof}

We now find the action of the Toeplitz operators with
quasi-homogeneous symbols on the canonical monomial basis. Note that
the following result corresponds to Lemma~3.3 from
\cite{NikolaiQuasi}.

\begin{lemma}\label{Lemma-khomogeneous}
  Let $k = (k_1, \dots, k_l) \in \Z_+^l$ be a partition of $n$, and
  let $p,q \in \Z_+^n$ be such that $p\cdot q = 0$ and $|p| =
  |q|$. Consider a bounded measurable $(1,k)$-quasi-homogeneous
  quasi-radial symbol that as a function $\C^n \rightarrow \C$ is
  written as $a\xi^p\overline{\xi}^q = a(r_1, \dots,
  r_l)\xi^p\overline{\xi}^q$.  Then, the Toeplitz operator
  $T_{a\xi^p\overline{\xi}^q}$ acts on monomials $z^{\alpha}$ with
  $\alpha\in \Z^n_+$ and $|\alpha|\leq m$ as follows
  \[
  T_{a\xi^p\overline{\xi}^q} z^{\alpha}=
  \begin{cases}
    \tilde{\gamma}_{a,k,p,q,m}(\alpha)z^{\alpha+p-q} &
    \text{ for } \alpha + p - q \in J_n(m) \\
    0 & \text{ for } \alpha + p - q \not\in J_n(m)
  \end{cases}
  \]
  where
  \begin{align}\label{spectrum-homoge}
    \tilde{\gamma}_{a,k,p,q,m}(\alpha) &=
     \frac{2^l (\alpha+p)! (n+m)!}{(\alpha+p-q)! (m-|\alpha+p-q|)!
      \prod_{j=1}^{l}(k_j-1+|\alpha_{(j)}+p_{(j)}|)!} \nonumber \\
    & \times \int_{\R^n_+}  a(r_1,\dots,r_l)(1+r^2)^{-(n+m+1)}
    \prod_{j=1}^{l}r_j^{2|\alpha_{(j)}+p_{(j)}-q_{(j)}|+2k_j-1} dr_j
  \end{align}
\end{lemma}
\begin{proof}
  Let $\alpha,\beta\in \Z^n_+$ satisfy $|\alpha|, |\beta| \leq
  m$. Then, we have
  \begin{align*}
    \langle T_{a\xi^p\overline{\xi}^q} z^{\alpha},
    z^{\beta}\rangle_{m}&=\langle a\xi^p\overline{\xi}^q z^{\alpha},
    z^{\beta}\rangle_{m}\\
    &=\frac{ (n+m)!}{\pi^n m!} \int_{\C^n}
    \frac{a(r_1,\dots,r_l) \xi^p\overline{\xi}^q
      z^{\alpha}\overline{z}^{\beta} dV(z)}{(1+|z_1|^2 + \dots +
      |z_n|^2)^{n+m+1}}.
  \end{align*}

  Applying the change of the variables $z_{(j)} = r_j \xi_{(j)}$,
  where $r_j \in [0, \infty)$ and $\xi_{(j)} \in \bS^{k_j}$, for $j =
  1,\dots, l$, this yields
  \begin{align}
    \label{Toeplitz-qh}
    \nonumber
    \langle a\xi^p\overline{\xi}^qz^{\alpha}, z^{\beta}\rangle_{m}
    &=\frac{(n+m)!}{\pi^n m!}
    \int_{\R^n_+} a(r_1,\dots,r_l)(1+r^2)^{-(n+m+1)} \\
    &\times
    \nonumber
    \prod_{j=1}^{l}r_j^{|\alpha_{(j)}|+|\beta_{(j)}|+2k_j-1} dr_j  \\
    \nonumber
    &\times \prod_{j=1}^{l} \int_{\bS^{k_j}}
    \xi^{\alpha_{(j)}+p_{(j)}}_{(j)}\overline{\xi}^{\beta_{(j)}+q_{(j)}}_{(j)}
    dS(\xi_{(j)}) \\
    &=\delta_{\alpha+p,\beta+q}\frac{2^l(\alpha+p)! (n+m)!}{ m!
      \prod_{j=1}^{l}(k_j-1+|\alpha_{(j)}+p_{(j)}|)!}\\
    \nonumber
    &\times \int_{\R^n_+} a(r_1,\dots,r_l)(1+r^2)^{-(n+m+1)} \\
    \nonumber
    &\times
    \prod_{j=1}^{l}r_j^{2|\alpha_{(j)}+p_{(j)}-q_{(j)}|+2k_j-1} dr_j
  \end{align}
  Observe that this expression is non zero if and only if $\beta =
  \alpha + p - q$, which a priori belongs to $J_n(m)$.  We conclude
  the result from the orthonormality of the basis defined in
  \eqref{eq:orth-basis}.
\end{proof}

\section{COMMUTATIVITY RESULTS FOR QUASI-HOMOGENEOUS SYMBOLS ON
  $\PC$} \label{sec:commutativity}
The results in this section show that the commuting identities proved
in \cite{NikolaiQuasi} for the unit ball $\B^n$ have corresponding
ones for the complex projective space $\PC$.

\begin{theorem}
  \label{thm:commutativity-(1,k)}
  Let $k = (k_1, \dots, k_l) \in \Z_+^l$ be a partition of $n$ and $p,
  q \in \Z_+^n$ a pair of orthogonal multi-indices.  Let $a_1, a_2 \in
  \mR_{(1,k)}$ be non identically zero and let $\xi^p \overline{\xi}^q
  \in \mH_{(1,k)}$. Then the Toeplitz operators $T_{a_1}$ and $T_{a_2
    \xi^p \overline{\xi}^q}$ commute on each weighted Bergman space
  $\mA^2_{m}(\PC)$ if and only if $|p_{(j)} | = |q_{(j)} |$ for each
  $j=1,\dots, l$.
\end{theorem}
\begin{proof}
  Let $\alpha\in J_n(m)$ be given. First note that if $\alpha + p - q
  \not\in J_n(m)$, then the Lemmas~\ref{Lemma-kradial} and
  \ref{Lemma-khomogeneous} imply that both
  $T_{a_1}T_{a_2\xi^p\overline{\xi}^q} z^{\alpha}$ and
  $T_{a_2\xi^p\overline{\xi}^q} T_{a_1} z^{\alpha}$ vanish. Hence, we
  can assume that $\alpha+p-q \in J_n(m)$.  Applying again
  Lemmas~\ref{Lemma-kradial} and \ref{Lemma-khomogeneous} we obtain
  \begin{align*}
    T_{a_1}T_{a_2\xi^p\overline{\xi}^q} z^{\alpha}
    &=\frac{2^l (\alpha+p)!
      (n+m)!}{(\alpha+p-q)! (m-|\alpha+p-q|)!
      \prod_{j=1}^{l}(k_j-1+|\alpha_{(j)}+p_{(j)}|)!} \\
    &\times \int_{\R^n_+}  a_2(r_1,\dots,r_l)(1+r^2)^{-(n+m+1)}
    \prod_{j=1}^{l}r_j^{2|\alpha_{(j)}+p_{(j)}-q_{(j)}|+2k_j-1} dr_j \\
    &\times \frac{2^l(n+m)!}{(m-|\alpha+p-q|)!
      \prod_{j=1}^{l}(k_j-1+|\alpha_{(j)}+p_{(j)}-q_{(j)}|)!}\\
    &\times \int_{\R^n_+}  a_1(r_1,\dots,r_l) (1+r^2)^{-(n+m+1)}
    \prod_{j=1}^{l}r_j^{2|\alpha_{(j)}+p_{(j)}-q_{(j)}|+2k_j-1} dr_j \\
    &\times  z^{\alpha+p-q}.
  \end{align*}
  And similarly, we have
  \begin{align*}
    T_{a_2\xi^p\overline{\xi}^q}  T_{a_1} z^{\alpha}
    &=\frac{2^l(n+m)!}{(m-|\alpha|)!
      \prod_{j=1}^{l}(k_j-1+|\alpha_{(j)}|)!}\\
    &\times  \int_{\R^n_+}  a_1(r_1,\dots,r_l)
    (1+r^2)^{-(n+m+1)}\prod_{j=1}^{l}r_j^{2|\alpha_{(j)}|+2k_j-1} dr_j\\
    &\times\frac{2^l (\alpha+p)! (n+m)!}{(\alpha+p-q)! (m-|\alpha+p-q|)!
      \prod_{j=1}^{l}(k_j-1+|\alpha_{(j)}+p_{(j)}|)!} \\
    &\times \int_{\R^n_+}  a_2(r_1,\dots,r_l)(1+r^2)^{-(n+m+1)}
    \prod_{j=1}^{l}r_j^{2|\alpha_{(j)}+p_{(j)}-q_{(j)}|+2k_j-1} dr_j \\
    &\times  z^{\alpha+p-q}
  \end{align*}
  From which we conclude that $T_{a_1}T_{a_2\xi^p\overline{\xi}^q}
  z^\alpha = T_{a_2\xi^p\overline{\xi}^q} T_{a_1} z^\alpha$ if and
  only if $|p_{(j)}|=|q_{(j)}|$ where $j=1,\dots, l$.
\end{proof}

If we assume that $|p_{(j)}|=|q_{(j)}|$ for all $j = 1, \dots, l$,
then the equations (\ref{spectrum-radial}) and (\ref{spectrum-homoge})
combine together to yield the following identity.
\begin{align}\label{spectrum-homoge-2}
  \nonumber
  \tilde{\gamma}_{a,k,p,q,m}(\alpha)  &= \frac{2^l (\alpha+p)!
    (n+m)!}{(\alpha+p-q)! (m-|\alpha+p-q|)!
    \prod_{j=1}^{l}(k_j-1+|\alpha_{(j)}+p_{(j)}|)!} \\
  \nonumber
  &\times \int_{\R^n_+}  a(r_1,\dots,r_l)(1+r^2)^{-(n+m+1)}
  \prod_{j=1}^{l}r_j^{2|\alpha_{(j)}|+2k_j-1} dr_j \\
  \nonumber
  &= \frac{ (\alpha+p)!
    \prod_{j=1}^{l}(k_j-1+|\alpha_{(j)}|)!}{(\alpha+p-q)!
    \prod_{j=1}^{l}(k_j-1+|\alpha_{(j)}+p_{(j)}|)!}
  \gamma_{a,k,m}(\alpha) \\
  &= \prod_{j=1}^{l}\left(\frac{ (\alpha_{(j)}+p_{(j)})!
      (k_j-1+|\alpha_{(j)}|)!}{(\alpha_{(j)}+p_{(j)}-q_{(j)})!
      (k_j-1+|\alpha_{(j)}+p_{(j)}|)!}\right)
  \gamma_{a,k,m}(\alpha)
\end{align}

As a consequence of the previous computations, we also obtain the
following very special property of Toeplitz operators with
quasi-homogeneous symbols.

\begin{corollary}\label{cor:T-quasihomog}
  Let $k = (k_1, \dots, k_l) \in \Z_+^l$ be a partition of $n$ and $p,
  q \in \Z_+^n$ a pair of orthogonal multi-indices such that $|p_{(j)}|
  = |q_{(j)}|$ for all $j = 1, \dots, l$. Then for each function $a
  \in \mR_{(1,k)}$, we have $T_{a}T_{ \xi^p \overline{\xi}^q}=T_{
    \xi^p \overline{\xi}^q}T_{a}=T_{ a\xi^p \overline{\xi}^q}.$
\end{corollary}

Consider $k=(k_1,\ldots,k_l)$ and a pair of multi-indices $p,q$ such
that $p\perp q$ and $|p_{(j)}|=|q_{(j)}|$ for $j=1,\ldots l$.  we
define
\[
\tilde{p}_{(j)}=(0,\ldots,0,p_{(j)},0,\ldots,0), \quad
\tilde{q}_{(j)}=(0,\ldots,0,q_{(j)},0,\ldots,0)
\]
where the only possibly non zero part is placed in the $j$-th
position. In particular, we have
$p=\tilde{p}_{(1)}+\ldots+\tilde{p}_{(l)}$ and
$q=\tilde{q}_{(1)}+\ldots+\tilde{q}_{(l)}$.  Now let
$T_j=T_{\xi^{\tilde{p}_{(j)}}\overline{\xi}^{\tilde{p}_{(j)}}}$ for
every $j = 1, \dots, l$. As a consequence of the previous computations
we obtain the following result.

\begin{corollary}
  The Toeplitz operators
  $T_j=T_{\xi^{\tilde{p}_{(j)}}\overline{\xi}^{\tilde{p}_{(j)}}}$, for
  $j=1,\ldots l$ mutually commute and
  \[
  \prod_{j=1}^lT_j=T_{\xi^p\overline{\xi}^q}
  \]
\end{corollary}

We now obtain a necessary and sufficient condition for two given
quasi-homogeneous symbols to determine Toeplitz operators that commute
with each other. For the next result we switch to the homogeneous
coordinates of $\PC$ to obtain a result whose statement is independent
of the choice of charts.

\begin{theorem}\label{thm:commutative-iff}
  Let $k = (k_1, \dots, k_l) \in \Z_+^l$ be a partition of $n$ so that
  $(1,k) \in \Z_+^{l+1}$ is a partition of $n+1$, and let $p, q, u, v
  \in \Z_+^{n+1}$ be multi-indices that satisfy the following properties
  \begin{itemize}
  \item $p \perp q$ and $u \perp v$,
  \item $|p_{(j)}| = |q_{(j)}|$ and $|u_{(j)}| = |v_{(j)}|$ for all $j
    = 0, \dots, l$.
  \end{itemize}
  We are assuming the enumerations $p = (p_0, p_1, \dots, p_n) =
  (p_{(0)}, p_{(1)}, \dots, p_{(l)})$, so that in particular $p_0 =
  p_{(0)} = 0$. Also, the corresponding properties hold for $q, u, v$
  as well.

  Let $a\xi^p\overline{\xi}^q, b\xi^u\overline{\xi}^v \in
  \mHR_{(1,k)}$ be corresponding $(1,k)$-quasi-homogeneous
  quasi-radial symbols on $\PC$, where $a, b \in \mR_{(1,k)}$ are
  measurable and bounded symbols. Then, the Toeplitz operators
  $T_{a\xi^p\overline{\xi}^q}$ and $T_{b\xi^u\overline{\xi}^v}$
  commute on each weighted Bergman space $\mA^2_{m}(\PC)$ if and only
  if for each $s=1,\ldots,n$ one of the following conditions hold
  \begin{enumerate}
  \item $p_s=q_s=0$
  \item $u_s=v_s=0$
  \item $p_s=u_s=0$
  \item $q_s=v_s=0$
  \end{enumerate}
\end{theorem}
\begin{proof}
  As remarked, the first entry of the multi-indices $p, q, u, v \in
  \Z_+^{n+1}$ vanishes. Hence, we will replace such multi-indices with
  their counterparts in $\Z_+^n$ obtained by removing the first
  entry. In particular, the multi-indices $p, q, u, v \in \Z_+^n$ so
  obtained are such that both pairs $(p,q)$ and $(u,v)$ satisfy the
  hypothesis of Theorem~\ref{thm:commutativity-(1,k)}. We will proceed
  to compute the composition of operators for the corresponding
  symbols in $\C^n$ as considered in Remark~\ref{rmk:Cn}.

  First, we observe that the quantities
  $T_{b\xi^u\overline{\xi}^v}T_{a\xi^p\overline{\xi}^q} z^{\alpha}$
  and $T_{a\xi^p\overline{\xi}^q}T_{b\xi^u\overline{\xi}^v}
  z^{\alpha}$ are always simultaneously zero or non zero. Hence, we
  compute such expressions for $\alpha \in J_n(m)$ assuming that both
  are non zero.

  By (\ref{spectrum-homoge-2}), we have the following expression
  \begin{align*}
    T_{b\xi^u\overline{\xi}^v}T_{a\xi^p\overline{\xi}^q} z^{\alpha}
    &= \frac{2^l (\alpha+p)! (n+m)!}{(\alpha+p-q)! (m-|\alpha|)!
      \prod_{j=1}^{l}(k_j-1+|\alpha_{(j)}+p_{(j)}|)!} \\
    &\times \int_{\R^n_+}  a(r_1,\ldots,r_l)(1+r^2)^{-(n+m+1)}
    \prod_{j=1}^{l}r_j^{2|\alpha_{(j)}|+2k_j-1} dr_j \\
    &\times \frac{2^l (\alpha+p-q+u)! (n+m)!}{(\alpha+p-q+u-v)!
      (m-|\alpha|)!  \prod_{j=1}^{l}(k_j-1+|\alpha_{(j)}+u_{(j)}|)!} \\
    &\times \int_{\R^n_+}  b(r_1,\ldots,r_l)(1+r^2)^{-(n+m+1)}
    \prod_{j=1}^{l}r_j^{2|\alpha_{(j)}|+2k_j-1} dr_j \\
    &\times z^{\alpha+p-q+u-v}.
  \end{align*}

  Similarly, we also have
  \begin{align*}
    T_{a\xi^p\overline{\xi}^q}T_{b\xi^u\overline{\xi}^v} z^{\alpha}
    &= \frac{2^l (\alpha+u)! (n+m)!}{(\alpha+u-v)! (m-|\alpha|)!
      \prod_{j=1}^{l}(k_j-1+|\alpha_{(j)}+u_{(j)}|)!} \\
    &\times \int_{\R^n_+}  b(r_1,\ldots,r_l)(1+r^2)^{-(n+m+1)}
    \prod_{j=1}^{l}r_j^{2|\alpha_{(j)}|+2k_j-1} dr_j \\
    &\times \frac{2^l (\alpha+u-v+p)! (n+m)!}{(\alpha+p-q+u-v)!
      (m-|\alpha|)!  \prod_{j=1}^{l}(k_j-1+|\alpha_{(j)}+p_{(j)}|)!} \\
    &\times \int_{\R^n_+}  a(r_1,\ldots,r_l)(1+r^2)^{-(n+m+1)}
    \prod_{j=1}^{l}r_j^{2|\alpha_{(j)}|+2k_j-1} dr_j \\
    &\times z^{\alpha+p-q+u-v}
  \end{align*}

  Therefore, we conclude that
  $T_{b\xi^u\overline{\xi}^v}T_{a\xi^p\overline{\xi}^q}
  z^{\alpha}=T_{a\xi^p\overline{\xi}^q}T_{b\xi^u\overline{\xi}^v}
  z^{\alpha}$ if and only if
  \[
  \frac{(\alpha + p)!(\alpha + p - q + u)!}{(\alpha + p - q)!} =
  \frac{(\alpha + u)!(\alpha + u - v + p)!}{(\alpha + u - v)!}.
  \]
  Finally, one can easily check that the latter identity holds for
  every $\alpha \in J_n(m)$ and $m \in \Z_+$ if and only if the
  conclusion of the statement holds. This proves the Theorem.
\end{proof}

Finally, we present one of our main results: the construction of a
commutative Banach algebra of Toeplitz operators on $\PC$. Our
construction is parallel to the one presented in \cite{NikolaiQuasi}.

\begin{definition}
  \label{def:Akh}
  Let $k = (k_1, \dots, k_l) \in \Z_+^l$ be a partition of $n$ and $h
  \in \Z_+^l$ be such that $1 \leq h_j \leq k_j - 1$, for all $j = 1,
  \dots, l$. We denote with $\mA_{k,h}$ the set of symbols $\varphi
  \in \mHR_{(1,k)}$ that satisfy the following properties.
  \begin{enumerate}
  \item The symbol $\varphi = a \xi^p \overline{\xi}^q$ is
    $(1,k)$-quasi-homogeneous quasi-radial on $\PC$, in other words,
    it is a function of the form
    \[
    \varphi([w]) = a(|w_0|, \dots, |w_n|) \prod_{j=0}^l
    \left(\frac{w_{(j)}}{|w_{(j)}|}\right)^{p_{(j)}}
    \left(\frac{\overline{w}_{(j)}}{|w_{(j)}|}\right)^{q_{(j)}}
    \]
    where $a$ is a degree $0$ homogeneous function and $p,q \in
    \Z_+^{n+1}$.
  \item The multi-indices $p,q$ are orthogonal and satisfy $|p_{(j)}|
    = |q_{(j)}|$ for all $j = 0, \dots, l$. In particular, $p_0 =
    p_{(0)} = q_0 = q_{(0)} = 0$.
  \item The multi-indices $p,q$ satisfy
    \[
    p_{k_0 + \dots + k_{j-1} + r} = 0, \quad
    q_{k_0 + \dots + k_{j-1} + s} = 0,
    \]
    whenever $1 \leq s \leq h_j < r \leq k_j$ and $j = 1, \dots,
    l$.
  \end{enumerate}
\end{definition}

\begin{remark}
  \label{rmk:Akh}
  Note that with the notation of Definition~\ref{rmk:Akh} we have $k_0
  = 1$. Then, for a symbol $\varphi \in \mA_{k,h}$ as described in
  Definition~\ref{def:Akh} we can write
  \begin{align*}
    \varphi([w]) &= a(|w_0|, \dots, |w_n|) \prod_{j=1}^l
    \left(\frac{w_{(j)}}{|w_{(j)}|}\right)^{p_{(j)}}
    \left(\frac{\overline{w}_{(j)}}{|w_{(j)}|}\right)^{q_{(j)}} \\
    &= a(1, |w_1|/|w_0|, \dots, |w_n|/|w_0|) \prod_{j=1}^l
    \left(\frac{w_{(j)}/w_0}{|w_{(j)}|/|w_0|}\right)^{p_{(j)}}
    \left(\frac{\overline{w}_{(j)}/\overline{w_0}}{|w_{(j)}|/|w_0|}\right)^{q_{(j)}}
    \\
    &= a(1, |z_1|, \dots, |z_n|) \prod_{j=1}^l
    \left(\frac{z_{(j)}}{|z_{(j)}|}\right)^{p_{(j)}}
    \left(\frac{\overline{z}_{(j)}}{|z_{(j)}|}\right)^{q_{(j)}}
  \end{align*}
  where $z_{(j)} = w_{(j)}/w_0$ are components of the inhomogeneous
  coordinates of $\PC$. Here we have used the degree $0$ homogeneity
  of $a$ and the property $|p_{(j)}| = |q_{(j)}|$ for all $j = 1,
  \dots, l$. This shows that the symbols in $\mA_{k,h}$ have a
  homogeneity property and that they can thus be expressed through
  inhomogeneous coordinates of $\PC$.
\end{remark}

An immediate consequence of Definition~\ref{def:Akh} is that, with the
notation of Theorem~\ref{thm:commutative-iff}, for every $s = 1,
\dots, n$ either (iii) or (iv) from the conclusion of such Theorem is
always satisfied for two given symbols in $\mA_{k,h}$. In particular,
we conclude the following result.

\begin{theorem}[Banach algebra of symbols with commuting operators]
  \label{thm:Banachalgebra}
  Let $k = (k_1, \dots, k_l) \in \Z_+^l$ be a partition of $n$ and $h
  \in \Z_+^l$ be such that $1 \leq h_j \leq k_j - 1$, for all $j = 1,
  \dots, l$. Then, the Banach algebra of Toeplitz operators generated
  by the symbols in $\mA_{k,h}$ is commutative on each weighted
  Bergman space $\mA^2_m(\PC)$.
\end{theorem}

\section{BUNDLES OF LAGRANGIAN FRAMES AND QUASI-HOMOGENEOUS SYMBOLS}
\label{sec:Lagrangian}
In the rest of this work, we fix a partition $k = (k_1, \dots, k_l)
\in \Z^l_+$ of $n$ and $h \in \Z^l_+$ that satisfy the conditions from
Definition~\ref{def:Akh}. We will provide in this section a geometric
construction on the projective space $\PC$ which is relevant to the
set of symbols $\mA_{k,h} \subset \mHR_{(1,k)}$.

Recall that the special linear, projective special groups on
$\C^{n+1}$, as well as their unitary counterparts, are defined,
respectively, as follows
\begin{align*}
  \SL(n+1,\C) &= \{ A \in M_{n+1}(\C) : \det(A) = 1 \} \\
  \PSL(n+1,\C) &= \SL(n+1,\C)/(\T I_{n+1} \cap \SL(n+1,\C)) \\
  \SU(n+1) &= \{ A \in M_{n+1}(\C) : A^* A = I_{n+1}, \det(A) = 1 \} \\
  \PSU(n+1) &= \SU(n+1) / (\T I_{n+1} \cap \SU(n+1)).
\end{align*}
It is easily seen that $\T I_{n+1} \cap \SL(n+1, \C) = \T I_{n+1} \cap
\SU(n+1)$, that it is cyclic of order $n+1$ and that it is precisely
the center of both $\SL(n+1, \C)$ and $\SU(n+1)$. We also note that
the natural quotient map $\SL(n+1, \C) \rightarrow \PSL(n+1, \C)$ is a
covering homomorphism whose (finite) kernel is $\T I_{n+1} \cap
\SL(n+1, \C)$. A corresponding observation holds for the covering
homomorphism obtained by restricting to $\SU(n+1) \rightarrow
\PSU(n+1)$. In what follows, for $A \in \SL(n+1, \C)$ we will denote
with $[A]$ its image in $\PSL(n+1, \C)$ with respect to this natural
covering homomorphism.

There is natural action of $\PSL(n+1, \C)$ on $\PC$ given by the
assigment
\[
([A], [w]) \mapsto [Aw],
\]
where $A \in \SL(n+1, \C)$ and $w \in \C^{n+1}$. Furthermore, it is
well known that this action realizes the group of biholomorphisms of
$\PC$. Also, the restriction of this action to $\SU(n+1)$ realizes the
connected component of the identity of the group of isometries for
$\PC$ with the Fubini-Study metric.

The following is an elementary result from the theory of compact
semisimple Lie groups (e.g.~see \cite{Helgason}).

\begin{proposition}
  \label{prop:SU-max-compact}
  Let us denote
  \[
  \mT = \{ [D] \in \PSU(n+1) : D \in \SU(n+1) \text{ is diagonal } \}.
  \]
  Then $\mT$ is isomorphic as a Lie group to $\T^n$ and it is a
  maximal Abelian subgroup of $\PSU(n+1)$. Furthermore, every maximal
  Abelian subgroup of $\PSU(n+1)$ is conjugate to $\mT$.
\end{proposition}

A remarkable fact about the commutative $C^*$-algebras of Toeplitz
operators introduced in \cite{GQV-disk}, \cite{QV-Ball1},
\cite{QV-Ball2} and \cite{QS-Proj} is that the correspoding sets of
symbols have a naturally associated maximal Abelian subgroup of the
group of isometries of the complex spaces supporting the Bergman
spaces. We now prove that a similar situation is also valid for the
sets of symbols $\mA_{k,h}$. This will be given as a subgroup of the
torus $\mT$.

\begin{theorem}[Torus associated to $\mA_{k,h}$]
  \label{thm:torus-mA}
  Let $k = (k_1, \dots, k_l) \in \Z^l_+$ be a partition of $n$ and let
  $h \in \Z^l_+$ be such that the conditions of
  Definition~\ref{def:Akh} hold. Consider the subgroup $\mT_k$ of
  $\mT$ defined by the condition
  \begin{itemize}
  \item For $M \in \mT$, we have $M \in \mT_k$ if and only if
    $\varphi(M[w]) = \varphi([w])$ for every $\varphi \in \mA_{k,h}$
    and $[w] \in \PC$.
  \end{itemize}
  Then, $\mT_k$ is a closed subgroup of $\mT$ isomorphic to
  $\T^l$. Furthermore, $M \in \mT_k$ if and only if we have
  \begin{equation}
    \label{eq:M-mT}
    M = \left[
      \begin{pmatrix}
        \overline{t}_1^{k_1} \cdots \overline{t}_l^{k_l} & 0 & \cdots & 0 \\
        0 & t_1 I_{k_1} & \cdots & 0 \\
        \vdots & \vdots & \ddots & \vdots \\
        0 & 0 & \cdots & t_l I_{k_l}
      \end{pmatrix}
    \right]
  \end{equation}
  for some $t_1, \dots, t_l \in \T$.
\end{theorem}
\begin{proof}
  Let us consider any given symbol $\varphi \in \mA_{k,h}$. In
  particular, we have
  \[
  \varphi([w]) = a(|w_0|, \dots, |w_n|) \prod_{j=0}^l
  \left(\frac{w_{(j)}}{|w_{(j)}|}\right)^{p_{(j)}}
  \left(\frac{\overline{w}_{(j)}}{|w_{(j)}|}\right)^{q_{(j)}}
  \]
  where $a$ and $p,q \in \Z_+^{n+1}$ satisfy the conditions from
  Definition~\ref{def:Akh}. In particular, as observed in
  Remark~\ref{rmk:Akh}, we have $p_0 = q_0 = 0$ and we further have
  \[
  \varphi([w]) = a(|w_0|, \dots, |w_n|) \prod_{j=1}^l
  \left(\frac{w_{(j)}}{|w_{(j)}|}\right)^{p_{(j)}}
  \left(\frac{\overline{w}_{(j)}}{|w_{(j)}|}\right)^{q_{(j)}}.
  \]
  Let $M \in \mT$ be an element defined a diagonal matrix in
  $\SU(n+1)$ with diagonal elements $t_0, \dots, t_n \in \T$ in that
  order. Hence, a direct computation shows that
  \[
  \varphi(M[w]) = a(|w_0|, \dots, |w_n|) \prod_{j=1}^l t_{(j)}^{p(j)}
  \overline{t}_{(j)}^{q(j)} \prod_{j=1}^l
  \left(\frac{w_{(j)}}{|w_{(j)}|}\right)^{p_{(j)}}
  \left(\frac{\overline{w}_{(j)}}{|w_{(j)}|}\right)^{q_{(j)}}.
  \]
  We conclude that, for our choice of $M$, we have $M \in \mT$ if and
  only if
  \begin{equation}
    \label{eq:prod1}
    \prod_{j=1}^l t_{(j)}^{p(j)} \overline{t}_{(j)}^{q(j)} = 1
  \end{equation}
  for every $p,q \in \Z_+^{n+1}$ that satisfy conditions (ii) and (iii)
  from Definition~\ref{def:Akh}.

  From the latter remarks it is easy to see that every $M$ of the form
  given by equation~\eqref{eq:M-mT} belongs to $\mT$. This is true
  since for $p,q$ satisfying condition (ii) from
  Definition~\ref{def:Akh} we have $|p_{(j)}| = |q_{(j)}|$ for all $j
  = 1, \dots, l$.

  Conversely, let us assume that $M \in \mT$, so that one has
  $\varphi(M[w]) = \varphi([w])$ for $\varphi$ as above and every $[w]
  \in \PC$. We will pick a particular choice of $p,q \in
  \Z_+^{n+1}$. Given $1 \leq j_0 \leq l$ choose $r,s$ such that $1 \leq
  r \leq h_j < s \leq k_j$ and define for $j = 1, \dots, l$
  \begin{align*}
    (p_{(j)})_i &=
    \begin{cases}
      1 \text{ if } j = j_0 \text{ and } i = r, \\
      0 \text{ otherwise},
    \end{cases} \\
    (q_{(j)})_i &=
    \begin{cases}
      1 \text{ if } j = j_0 \text{ and } i = s, \\
      0 \text{ otherwise}.
    \end{cases}
  \end{align*}
  We also choose $a = 1$. Then, it is easy to check that the
  corresponding symbol $\varphi$ belongs to $\mA_{k,h}$. For this
  symbol, the condition given by equation~\eqref{eq:prod1} reduces to
  \[
  (t_{(j_0)})_r = (t_{(j_0)})_s
  \]
  for all $r,s$ satisfying $1 \leq r \leq h_j < s \leq k_j$ for our
  arbitrarily given $1 \leq j_0 \leq l$. In particular, the diagonal
  entries of the matrix $D$ such that $M = [D]$ are all the same on
  each one of the index intervals defined by the partition $k$. Since
  $\det(D) = 1$ we thus conclude that $M$ is the form shown in
  equation~\eqref{eq:M-mT}.

  Hence we have proved the last claim of the statement. From this it
  clearly follows that $\mT_k$ is closed and isomorphic to $\T^l$, for
  which one uses that the canonical map $\SU(n+1) \rightarrow
  \PSU(n+1)$ is a covering homomorphism with finite kernel.
\end{proof}

\begin{remark}
  \label{rmk:mT-ind-of-h}
  It is interesting to note that the group $\mT_k$ associated to the
  set of symbols $\mA_{k,h}$ only depends on the partition $k$ and not
  on the multi-index $h$.
\end{remark}

As remarked above, the commutative $C^*$-algebras of Toeplitz
operators introduced in \cite{GQV-disk}, \cite{QV-Ball1},
\cite{QV-Ball2} and \cite{QS-Proj} came with a natural maximal Abelian
subgroup. Furthermore, such subgroup allowed to introduce a foliation
with a distinguished symplectic geometry. We will now show that the
group $\mT_k$ can be used to obtain a similar geometric construction
associated to the set of symbols $\mA_{k,h}$. Such construction will
be performed on a stratification of the projective space $\PC$.

First, we consider the complexification of the torus $\mT_k$ as
described by the conclusion of Theorem~\ref{thm:torus-mA}. More
precisely, we denote
\[
\mT^\C_k = \left\{ \left[
    \begin{pmatrix}
      z_0 & 0 & \cdots & 0 \\
      0 & z_1 I_{k_1} & \cdots & 0 \\
      \vdots & \vdots & \ddots & \vdots \\
      0 & 0 & \cdots & z_l I_{k_l}
    \end{pmatrix}
  \right] : z_0, \dots, z_l \in \C^*,\quad z_0 z_1^{k_1}\cdots
  z_l^{k_l} = 1 \right\}.
\]
In other words, $\mT^\C_k$ consists of the elements $[A] \in \PSL(n+1,
\C)$ where $A$ is a block diagonal matrix whose blocks along the
diagonal are $(z_0, z_{1} I_{k_j}, \dots, z_{k_l} I_{k_l})$ for some
$z_0, \dots, z_{k_l} \in \C^*$ such that $z_0 z_1^{k_1}\cdots
z_l^{k_l} = 1$.

Hence, $\mT^\C_k$ is a subgroup of $\PSL(n+1, \C)$ isomorphic to
$\C^{*l}$ containing $\mT_k$. In particular, $\mT^\C_k$ acts
biholomorphically on the projective space $\PC$. But, as it is easily
seen, the $\mT^\C_k$-orbits in $\PC$ have varying dimensions. We deal
with this situation considering a stratification associated to the
partition $k$.

For the given partition $k = (k_1, \dots, k_l)$ of $n$, consider the
decomposition $\C^{n+1} = \C \times \C^{k_1} \times \dots \times
\C^{k_l}$ that induces the corresponding decomposition $w = (w_0 =
w_{(0)}, w_{(1)}, \dots, w_{(l)})$ for every $w \in \C^{n+1}$. With
this notation we define
\begin{align*}
  V_0 &= \{ [w] \in \PC : w_0, w_{(1)}, \dots, w_{(l)} \not= 0 \} \\
  V_j &= \{ [w] \in \PC : w_{(j)}, \dots, w_{(l)} \not= 0,\quad w_0 =
  0, \dots, w_{(j-1)} = 0 \}
\end{align*}
where $j = 1, \dots, l$. Note that $V_0, \dots, V_l$ provides a
partition of $\PC$ into smooth complex quasi-projective
subvarieties. Furthermore, a direct inspection shows that $V_0$ is the
largest subset of $\PC$ where $\mT^\C_k$ acts freely. Recall that a
group acting on a set is said to do so freely if the stabilizers are
all trivial, in which case one also says that the action is free.

Consider the following finite sequence of subgroups of
$\mT^\C_k$. First we let $G_0 = \mT^\C_k$. And for every $j = 1,
\dots, l$ we define $G_j$ as the subgroup of $G_0$ which consists of
the elements $[A] \in \PSL(n+1, \C)$ where $A$ is a block diagonal
matrix whose blocks along the diagonal are
\[
(z_0, I_{k_1}, \dots,
I_{k_{j-1}}, z_j I_{k_j}, \dots, z_l I_{k_l})
\]
for some $z_0, z_j, \dots, z_l \in \C^*$ such that $z_0
z_j^{k_j}\cdots z_l^{k_l} = 1$. In particular, we clearly have that
$G_j$ is isomorphic to $\C^{*(l-j+1)}$ for every $j = 0, \dots, l$. We
will also consider for $j = 0, \dots, l$ the following compact groups
which can be thought as real forms of the groups $G_j$
\[
\mT_j = G_j \cap \mT_k.
\]
In particular, one can easily check that $\mT_j \simeq \T^{l-j+1}$,
for every $j = 0, \dots, l$, and that $\mT_0 = \mT_k$.

We now state the following easy to prove result. The main point is to
observe that $V_j \cup \dots \cup V_l$ is the subvariety defined by
the homogeneous equations $w_0 = 0, \dots, w_{(j-1)} = 0$. The rest
follows from this or it can be verified directly.

\begin{lemma}
  \label{lem:strat-free}
  For the partition $\PC = V_0 \cup \dots \cup V_l$ and the groups
  $G_0, \dots, G_l$ defined above, the following properties are
  satisfied for every $j = 0, \dots, l$.
  \begin{enumerate}
  \item The subset $V_j \cup \dots \cup V_l$ is a closed smooth
    projective subvariety of $\PC$.
  \item The smooth variety $V_j$ is open in $V_j \cup \dots \cup V_l$.
  \item The group $G_j$ leaves invariant the set $V_j \cup \dots \cup
    V_l$.
  \item The smooth variety $V_j$ is the largest subset of $V_j \cup
    \dots \cup V_l$ where $G_j$ acts freely.
  \end{enumerate}
\end{lemma}

We now construct a finite collection of principal fiber bundles whose
total spaces are the subvarities $V_j$ and whose structure groups are
subgroups of $G_j$. We refer to \cite{KNI} for the notion of principal
fiber bundle. This yields a partition of $\PC$ into principal fiber
bundles associated to the partition $k$, and it thus provides a
geometric structure for the set of symbols $\mA_{k,h}$.

\begin{theorem}[Principal bundles associated to $\mA_{k,h}$]
  \label{thm:partitions-bundles}
  For $k = (k_1, \dots, k_l) \in \Z^l_+$ a partition of $n$, consider
  the subvarieties $V_0, \dots, V_l$ of $\PC$ and the subgroups $G_0,
  \dots, G_l$ of $\mT^\C_k$ as defined above. Then, the following
  property is satisfied for every $j = 0, \dots, l$.
  \begin{itemize}
  \item The quotient space $G_j\backslash V_j$ is a smooth complex
    manifold so that the natural quotient map $V_j \rightarrow
    G_j\backslash V_j$ is a smooth complex principal fiber bundle with
    structure group $G_j \simeq \C^{*(l-j+1)}$. In particular, every
    $G_j$-orbit is a complex submanifold of $\PC$.
  \end{itemize}
\end{theorem}
\begin{proof}
  It is well known that a free proper action of a Lie group provides a
  quotient map that defines a principal fiber bundle (see for example
  \cite{GHL}). By virtue of Lemma~\ref{lem:strat-free} it is enough to
  show that the action of $G_j$ on $V_j$ is proper.

  We recall that a group $G$ acts properly on a manifold $V$ if for
  every compact subset $K \subset V$ the set
  \[
  \{ g \in G : gK \cap K \not= \emptyset \}
  \]
  is relatively compact in $G$.

  Choose $K \subset V_j$ a compact subset. Then, there exists
  $\widehat{K} \subset \C^{n+1} \setminus \{0\}$ a compact subset such
  that $K = \{ [w] : w \in \widehat{K} \}$. If we denote with $\pi$
  the canonical projection $\C^{n+1} \setminus \{0\} \rightarrow \PC$,
  then we can choose, for example, $\widehat{K} = \pi^{-1}(K) \cap
  \bS^{n+1}$. Note that from the definition of $V_j$ we have $w_{(j)},
  \dots, w_{(l)} \not= 0$ for every $w \in \widehat{K}$. And by
  compactness of $\widehat{K}$, there exists constants $c_1, c_2 > 0$
  such that
  \begin{equation}
    \label{eq:boundsK}
    c_1 \leq |w_{(r)}| \leq c_2
  \end{equation}
  for every $w \in \widehat{K}$ and $r = j, \dots, l$.

  Let $[A] \in G_j$ be such that $[A] K \cap K \not= \emptyset$. Hence
  we can assume the following
  \begin{enumerate}
  \item the matrix $A$ is a diagonal block matrix whose blocks along
    the diagonal are
    \[
    (z_0, I_{k_1}, \dots, I_{k_{j-1}}, z_j I_{k_j}, \dots, z_l
    I_{k_l})
    \]
    for some $z_0, z_j, \dots, z_l \in \C^*$ such that $z_0
    z_j^{k_j}\cdots z_l^{k_l} = 1$,
  \item there exist $w, w' \in \widehat{K}$ and $\lambda \in \C^*$
    such that for every $r = j, \dots l$ we have
    \[
    \lambda z_r w_{(r)} = w'_{(r)}.
    \]
  \end{enumerate}
  We conclude from \eqref{eq:boundsK} and (ii) that for every $r = j,
  \dots, l$ we have
  \begin{equation}
    \label{eq:bounds-lambdaz}
    d_1 \leq |\lambda z_r| \leq d_2
  \end{equation}
  where $d_1 = c_1/c_2$ and $d_2 = c_2/c_1$. This and (i) implies in
  turn that
  \[
  d_1^{\sigma_j} \leq |\lambda|^{\sigma_j} = |\lambda
  z_0|\prod_{r=j}^l |\lambda z_r|^{k_r} \leq d_2^{\sigma_j},
  \]
  where $\sigma_j = k_j + \dots + k_l + 1$. Hence we have $d_1 \leq
  |\lambda| \leq d_2$, which together with
  Equation~\eqref{eq:bounds-lambdaz} yields for $r = j, \dots, l$ the
  estimate
  \[
  \frac{d_1}{d_2} \leq |z_r| \leq \frac{d_2}{d_1}.
  \]
  Since $d_1, d_2 > 0$, the latter inequalities define a compact
  subset of $G_j \simeq \C^{*(l-j+1)}$. This completes the proof of
  the properness of the $G_j$-action on $V_j$.
\end{proof}

We recall the definition of Lagrangian frame introduced in
\cite{QV-Ball1, QV-Ball2, QV-Reinhardt}. We refer to the latter works
for further details on the notions involved.

\begin{definition}
  On a K\"ahler manifold $N$, a Lagrangian frame is a pair $(\OO, \F)$
  of smooth foliations that satisfy the following properties.
  \begin{itemize}
  \item Both foliations are Lagrangian. In other words, the leaves of
    both foliations are Lagrangian submanifolds of $N$.
  \item If $L_1$ and $L_2$ are leaves of $\OO$ and $\F$, respectively,
    then $T_xL_1 \perp T_xL_2$ at every $x \in L_1\cap L_2$.
  \item The foliation $\OO$ is Riemannian. In other words, the
    Riemannian metric of $N$ is invariant by the leaf holonomy of
    $\OO$.
  \item The foliation $\F$ is totally geodesic. In other words, its
    leaves are totally geodesic submanifolds of $N$.
  \end{itemize}
  We will refer to $\OO$ and $\F$ as the Riemannian and totally
  geodesic foliations, respectively, of the Lagrangian frame.
\end{definition}

The next result shows that the fibers of the submersion of the
principal bundles $V_j \rightarrow G_j\backslash V_j$ carry Lagrangian
frames naturally associated to the symbols $\mA_{k,h}$. We observe and
emphasize that the K\"ahler structure considered on any complex
submanifold of $\PC$ is the one obtained by restriction of the
Fubini-Study metric to the corresponding submanifold.

\begin{theorem}[Lagrangian frames associated to $\mA_{k,h}$]
  \label{thm:Lframes}
  For $k = (k_1, \dots, k_l) \in \Z^l_+$ a partition of $n$, consider
  the subvarieties $V_0, \dots, V_l$ of $\PC$ and the subgroups $G_0,
  \dots, G_l$ of $\mT^\C_k$ as defined above. Then, for every $j = 0,
  \dots, l$ and for every fiber $F$ of the principal bundle $V_j
  \rightarrow G_j\backslash V_j$, the following properties hold.
  \begin{enumerate}
  \item The action of $\mT_j$ restricted to $F$ defines a Riemannian
    foliation $\OO_F$ on whose leaves every symbol that belongs to
    $\mA_{k,h}$ is constant.
  \item The vector bundle $T\OO_F^\perp$ defined as the orthogonal
    complement of $T\OO_F$ inside of $TF$ is integrable to a totally
    geodesic foliation $\iO_F$.
  \item The pair $(\OO_F, \iO_F)$ is a Lagrangian frame of the complex
    manifold $F$ for the K\"ahler structure on $F$ inherited from
    $\PC$.
  \end{enumerate}
\end{theorem}
\begin{proof}
  Let us fix $j$ and $F$ as described in the statement.

  By Lemma~\ref{lem:strat-free} the group $G_j$ acts freely on $F$,
  because $F \subset V_j$. Furthermore, note that by
  Theorem~\ref{thm:partitions-bundles} and our choices, $F$ is in fact
  a free $G_j$-orbit. In particular, $\dim_\C F = \dim_\C G_j =
  l-j+1$. Also, since $\mT_j$ is a subgroup of $G_j$, we conclude that
  $\mT_j$ acts freely on $F$, thus defining a foliation $\OO_F$ whose
  leaves have (real) dimension $\dim \mT_j = l-j+1$.

  We now recall that the $\mT_j \subset \mT$ and that, by
  Proposition~\ref{prop:SU-max-compact}, the latter acts by
  isometries. This implies that the $\mT_j$-action on $F$ is isometric
  as well. This last property implies that the foliation $\OO_F$ is
  Riemannian (see, for example, \cite{QV-Reinhardt}). Also, since
  $\mT_j \subset \mT$, Theorem~\ref{thm:torus-mA} implies that the
  symbols belonging to $\mA_{k,h}$ are $\mT_j$-invariant and so
  constant on the leaves of $\OO_F$. This proves (i).

  It is known that the $\mT$-action on $\PC$ has isotropic orbits: the
  $\mT$-orbits are null with respect to the symplectic form of
  $\PC$. This has been verified in \cite{QS-Proj} for $V_0 \subset
  \PC$. More precisely, the latter claim is the content of
  Theorems~6.7 and 6.8 from \cite{QS-Proj}, whose proof is a direct
  consequence of Theorems~2.1 and 3.1 from \cite{NT}. We now observe
  that the results found in \cite{NT} are in fact stated for arbitrary
  orbits of the maximal compact subgroup $\mT$. This can be easily
  applied to conclude that the every $\mT$-orbit is in fact an
  isotropic submanifold of $\PC$. Next, the elements $[w] \in V_j$ are
  characterized by the conditions
  \[
  w_{(j)}, \dots, w_{(l)} \not= 0,\quad w_0 = 0, \dots, w_{(j-1)} = 0,
  \]
  from which it is easily seen that for every $t \in \mT$ there exists
  $t' \in \mT_j$ such that $t[w] = t'[w]$ (e.g.~define the components
  of $t'$ as $1$ at the positions where $w$ vanishes). This implies
  that the $\mT$-orbits in $V_j$ are precisely $\mT_j$-orbits, thus
  that the foliation $\OO_F$ has isotropic leaves. Since the real
  dimension of such leaves is $\dim \mT_j = l-j+1 = \dim_\C F$, we
  conclude that the foliation $\OO_F$ is Lagrangian in $F$.

  Since the orthogonal complement of a Riemannian foliation is totally
  geodesic (see, for example, \cite{QV-Reinhardt}), to prove (ii) and
  (iii) it suffices to show that $T\OO_F^\perp = iT\OO_F$ is
  integrable.

  To prove the integrability of $iT\OO_F$ let us consider $\{ X_r : r
  = 1, \dots, l-j+1 \}$ a base for the Lie algebra of $\mT_j$. The
  $\mT_j$-action on $F$ induces a family of vector fields $\{ X_r^* :
  r = 1, \dots, l-j+1 \}$ on $F$ characterized as those having flows
  given by $\{ \exp(X_r) : r = 1, \dots, l-j+1 \}$, respectively. We
  refer to \cite{Helgason} for further details on this
  construction. Since the $\mT_j$-action on $F$ is free, it follows
  that the vector fields $\{ X_r^* : r = 1, \dots, l-j+1 \}$ are
  linearly independent at every point of $F$, thus defining a
  generating set for $T\OO_F$ at every point of $F$. Furthermore,
  since $\mT_j$ is Abelian $[X^*_r, X^*_s] = 0$ for every $r,s$.

  On the other hand, for $J$ the complex structure, the set of vector
  fields $\{ JX_r^* : r = 1, \dots, l-j+1 \}$ yields a generating set
  for $iT\OO_F$ at every point of $F$.

  \emph{Claim:} For every $r$, the vector fields $X^*_r$ and $JX^*_r$
  are holomorphic, i.e.~they integrate to holomorphic local flows.

  To prove the claim, we first note that the vector fields $X^*_r$
  integrate by a definition to local flows which are $1$-parameter
  subgroups of the $\mT_j$-action. Since the latter is holomorphic, we
  conclude that the vector fields $X^*_r$ are holomorphic.

  We now consider the vector fields $JX^*_r$. First we recall that
  $\mT_j \subset \mT \subset \PSL(n+1,\C)$, which implies that
  \[
  X_r \in \Lie(\mT_j) \subset \Lie(\PSL(n+1,\C)) = \Lie(\SL(n+1,\C)) =
  \sli(n+1,\C).
  \]
  Let us denote with $\pi : \C^{n+1}\setminus \{0\} \rightarrow \PC$
  the canonical quotient map. From our definitions, it is clear that
  the natural $\SL(n+1,\C)$-action on $\C^{n+1}\setminus \{0\}$
  descends through $\pi$ to the $\PSL(n+1,\C)$-action on $\PC$. Hence,
  if we denote with $\widehat{X}_r$ the vector field on
  $\C^{n+1}\setminus \{0\}$ induced from $X_r$ and the
  $\SL(n+1,\C)$-action, then it is clear that $\widehat{X}_r$ and
  $X^*_r$ are $\pi$-related; in other words, we have
  \[
  X^*_r = d\pi(\widehat{X}_r).
  \]
  And since $\pi$ is holomorphic, we also have
  \[
  JX^*_r = d\pi(J\widehat{X}_r).
  \]
  In particular, to prove that $JX^*_r$ is holomorphic, it is enough
  to prove that $J\widehat{X}_r$ integrates to a holomorphic local
  flow.

  Since $X_r \in \Lie(\mT) \subset \sli(n+1,\C)$, its matrix is
  diagonal and pure imaginary. If $ic_0, \dots, ic_n$ are its diagonal
  elements, then $\widehat{X}_r$ integrates to the flow on
  $\C^{n+1}\setminus \{0\}$ given by
  \begin{align*}
    \R \times \C^{n+1}\setminus \{0\} &\rightarrow \C^{n+1}\setminus
    \{0\} \\
    (\theta, w) &\mapsto (e^{ic_0\theta}w_0, \dots, e^{ic_n\theta}w_n).
  \end{align*}
  In particular, we have on $\C^{n+1}\setminus \{0\}$ that
  \[
  \widehat{X}_r|_w = (ic_0w_0, \dots, ic_nw_n),
  \]
  and so that
  \[
  J\widehat{X}_r|_w = (c_0w_0, \dots, c_nw_n).
  \]
  The latter vector field clearly integrates to the flow on
  $\C^{n+1}\setminus \{0\}$ given by
  \begin{align*}
    \R \times \C^{n+1}\setminus \{0\} &\rightarrow \C^{n+1}\setminus
    \{0\} \\
    (\theta, w) &\mapsto (e^{c_0\theta}w_0, \dots, e^{c_n\theta}w_n),
  \end{align*}
  which is clearly holomorphic. This implies the holomorphicity of
  $JX^*_r$ and thus completes the proof of the Claim.

  Once the above is given, we have for every $r,s$
  \[
  [JX^*_r, JX^*_s] = J[X^*_r, JX^*_s] = J^2[X^*_r, X^*_s] = 0.
  \]
  Here we have used in the first and second identities the fact that
  $JX^*_r$ and $X^*_r$, respectively, define Lie derivatives that
  commute with $J$; the latter is a consequence of the fact that both
  vector fields are holomorphic (see \cite{KNII}). Thus, we have
  proved that the bundle $iT\OO_p$ has a set of sections that generate
  the fibers and commute pairwise. Hence, the integrability of
  $iT\OO_p$ follows from Frobenius Theorem.
\end{proof}

Finally, we prove that a suitable complement of $\mT_j$ in $G_j$ acts
by symmetries of the bundle obtained in
Theorem~\ref{thm:partitions-bundles}.

As above, consider a partition $k = (k_1, \dots, k_l) \in \Z^l_+$ of
$n$. For every $j = 0, \dots, l$ define the group $H_j$ as the
subgroup of $\PSL(n+1, \C)$ which consists of the classes $[D]$ where
$D$ is a block diagonal matrix with diagonal entries
\[
(1, I_{k_1}, \dots, I_{k_{j-1}}, D_j, \dots, D_l).
\]
where $D_r$ is a $k_r \times k_r$ diagonal matrix such that $\det(D_r)
= 1$.

Following our previous notation, we will also denote
\[
\mT^\C = \{ [D] : D \in \SL(n+1,\C) \text{ is diagonal}\}
\]
Then, the next result is a simple exercise.

\begin{lemma}
  \label{lem:Gj-Hj}
  For $k = (k_1, \dots, k_l) \in \Z^l_+$ a partition of $n$, and for
  every $j = 0, \dots, l$, the map
  \begin{align*}
    G_j \times H_j &\rightarrow \mT^\C \\
    ([D],[E]) &\mapsto [DE]
  \end{align*}
  is an isomorphism of Lie groups.
\end{lemma}

Let us fix $j = 0, \dots, l$. As a subgroup of $\PSL(n+1,\C)$, the
group $H_j$ clearly acts holomorphically on $\PC$. Also, from the
definition of $V_j$, the $H_j$-action clearly satisfies the following
properties
\begin{itemize}
\item $H_j$ leaves invariant $V_j$.
\item The $H_j$-action on $V_J$ is free.
\end{itemize}
Furthermore, by Lemma~\ref{lem:Gj-Hj} the actions of the groups $G_j$
and $H_j$ commute with each other. In particular, the $H_j$-action
maps every $G_j$-orbits in $V_j$ onto some $G_j$-orbit in $V_j$. This
construction allows us to obtain the following result. We refer to
\cite{KNI} for the definition of an automorphism of a principal
bundle.

\begin{theorem}
  \label{thm:bundleaut}
  For $k = (k_1, \dots, k_l) \in \Z^l_+$ a partition of $n$, and for
  every $j = 0, \dots, l$ the $H_j$-action on $V_j$ descends to an
  action on $G_j\backslash V_j$ so that the $H_J$-action defines an
  automorphism of the principal bundle
  \[
  \pi_j : V_j \rightarrow G_j \backslash V_j.
  \]
  In particular, the $H_j$-action maps fibers onto fibers (of
  $\pi_j$). Furthermore, the $H_j\cap\mT$-action maps the Lagrangian
  frames (as defined by Theorem~\ref{thm:Lframes}) into the
  corresponding Lagrangian frames.
\end{theorem}
\begin{proof}
  By the above, only the last part requires justification. But such
  claim is also clear since the leaves of the Riemannian foliation
  of the Lagrangian frames are defined as orbits of a subgroup of
  $G_j$, the totally geodesic foliation as its orthogonal complement
  and because the $H_j\cap\mT$-action is isometric.
\end{proof}

\end{document}